\documentclass[12pt,reqno]{amsart}
\usepackage{palatino}
\usepackage{amsmath,amssymb,amsxtra,color,calligra,mathrsfs}
\usepackage{xcolor}
\colorlet{mdtRed}{red!50!black}
\definecolor{dblue}{rgb}{0,0,.6}
\usepackage[colorlinks]{hyperref}
\hypersetup{linkcolor=mdtRed,citecolor=dblue,filecolor=dullmagenta,urlcolor=mdtRed}
\usepackage[all]{xy}
\usepackage{tikz,tikz-cd,tkz-graph,enumerate}
\usetikzlibrary{matrix,arrows,decorations.pathmorphing}

\DeclareMathOperator{\End}{\textnormal{End}}

\DeclareMathOperator{\ad}{\textnormal{ad}}
\DeclareMathOperator{\Pic}{\textnormal{Pic}}
 
\DeclareMathOperator{\Higgs}{\textnormal{Higgs}}
\DeclareMathOperator{\Lie}{\textnormal{Lie}}

\newtheorem{theorem}{Theorem}[section]
\newtheorem{lemma}[theorem]{Lemma} 
\newtheorem{proposition}[theorem]{Proposition}
\newtheorem{corollary}[theorem]{Corollary}

\theoremstyle{definition}
\newtheorem{definition}[theorem]{Definition}

\numberwithin{equation}{section} 

\begin{document}

\baselineskip=15.5pt 

\title[Picard group and fundamental group of moduli spaces]{Picard group 
and fundamental group of the moduli of Higgs bundles on curves}

\author{Sujoy Chakraborty} 
\address{School of Mathematics, Tata Institute of Fundamental Research, 
Homi Bhabha Road, Colaba, Mumbai 400005, India.}
\email{sujoy@math.tifr.res.in}

\author{Arjun Paul} 
\address{School of Mathematics, Tata Institute of Fundamental Research, 
Homi Bhabha Road, Colaba, Mumbai 400005, India.}
\email{apmath90@math.tifr.res.in}



%\date{\today}
%\date{June 12, 2018}

\begin{abstract} 
	Let $X$ be an irreducible smooth projective curve of genus $g \geq 2$ over $\mathbb{C}$. Let $G$ be a 
	connected reductive affine algebraic group over $\mathbb{C}$. Let $\mathrm{M}_{G, {\rm Higgs}}^{\delta}$ 
	be the moduli space of semistable principal $G$--Higgs bundles on $X$ of topological type 
	$\delta \in \pi_1(G)$. In this article, we compute the fundamental group and Picard group of 
	$\mathrm{M}_{G, {\rm Higgs}}^{\delta}$. 
\end{abstract}

\maketitle 

\section{Introduction}
Let $X$ be an irreducible smooth projective curve of genus $g \geq 2$ over $\mathbb{C}$. Let $G$ be a connected 
reductive affine algebraic group over $\mathbb{C}$. The topological types of holomorphic principal $G$--bundles 
on $X$ are parametrized by $\pi_1(G)$. 
Let $\mathrm{M}_G^{\delta}$ be the moduli space of semistable holomorphic principal $G$--bundles on $X$ of 
topological type $\delta \in \pi_1(G)$. This is an irreducible normal complex projective variety (generally non-smooth) of 
dimension $(g-1)\cdot\dim(G) + \dim(Z(G))$, where $Z(G)$ is the center of the group $G$. Geometry of moduli spaces 
of bundles over projective curves is an important topic to study in algebraic geometry. Picard group of 
$\mathrm{M}_G^{\delta}$ was studied in \cite{KN} for simply connected semisimple complex affine algebraic groups. 
Later A. Beauville, Y. Laszlo and C. Sorger studied the Picard group of $\mathrm{M}_G^{\delta}$ case by case for 
almost all classical semisimple complex affine algebraic groups (see \cite{BLS}). The case of reductive groups was 
studied in \cite{BH} for moduli stacks. 
The fundamental group of moduli space (and stack) of $G$--bundles was studied in \cite{BMP}. The case of 
moduli of $G$--Higgs bundles was open. In this article we study the fundamental group and Picard group of 
the moduli spaces of semistable $G$--Higgs bundles over $X$. 

Topological type of a holomorphic principal $G$--Higgs bundle on $X$ is 
defined by the topological type of the underlying principal $G$--bundle on $X$. Let 
$\mathrm{M}_{G, \Higgs}^{\delta}$ be the moduli space of semistable holomorphic principal $G$--Higgs bundles on 
$X$ of topological type $\delta \in \pi_1(G)$. The space $\mathrm{M}_{G, \Higgs}^{\delta}$ is nonempty and 
connected, for all $\delta \in \pi_1(G)$, \cite[Theorem 1.1, p.~791]{GO}. 
Let  $\mathrm{M}_G^{\delta}$ be the  moduli space of semistable principal $G$--bundles on $X$ of topological 
type $\delta \in \pi_1(G)$. For any $\mathbb{C}$--scheme $Z$, we denote by $\Pic(Z)$ (resp., $\pi_1(Z)$) 
the Picard group (resp., fundamental group) of $Z$. 

Assume that either $g \geq 3$, or there is no homomorphism of $G$ onto $\mathrm{PSL}_2(\mathbb{C})$ whenever $g = 2$.
Then we have the following 
\begin{theorem}
	There are isomorphisms 
	$\Pic(\mathrm{M}_{G, \Higgs}^{\delta}) \cong \Pic(\mathrm{M}_{G}^{\delta})$ and 
	$\pi_1(\mathrm{M}_{G, \Higgs}^{\delta}) \cong \pi_1(\mathrm{M}_{G}^{\delta})$. 
\end{theorem}
Therefore, it follows from \cite[Theorem 1.1]{BMP} that 
$\pi_1(\mathrm{M}_{G, \Higgs}^{\delta}) \cong \mathbb{Z}^{2gd}$, where $d = \dim Z(G)$ is the dimension 
of the center of $G$. In particular, $\mathrm{M}_{G, \Higgs}^{\delta}$ is simply connected, whenever 
$G$ is connected and semisimple. Therefore, together with the results of \cite{KN, BLS}, the above theorem 
determines Picard group of $\mathrm{M}_{G, \Higgs}^{\delta}$ for essentially all classical semisimple 
complex affine algebraic groups. 

We further remark that, our method of determining the fundamental group and Picard group generalize to the 
case of moduli stack of $G$--Higgs bundles over $X$. 

\section{Preliminaries} 
\subsection{Some results on locally trivial fibrations} 
Let $Z$ be a smooth quasi-projective variety over $\mathbb{C}$. We first recall some well-known 
results related to fundamental groups and Picard groups of vector bundles over $Z$. 
Let $p : E \longrightarrow Z$ be a vector bundle of finite rank over $Z$. 
\begin{lemma}\label{lem-1}
	The natural homomorphism $p^* : \Pic(Z) \longrightarrow \Pic(E)$ is an isomorphism. 
\end{lemma}

\begin{proof}
	Let $n$ be the rank of the vector bundle $E$ over $Z$. Note that the fibers of $p$ are affine $n$--spaces 
	$\mathbb{A}_{\mathbb{C}}^n$. We have a short exact sequence of groups 
	$$\Pic(Z) \stackrel{p^*}{\longrightarrow} \Pic(E) \stackrel{}{\longrightarrow} \Pic(\mathbb{A}_{\mathbb{C}}^n)\,,$$ 
	where the second map is given by the restriction of a line bundle to the fiber $\mathbb{A}_{\mathbb{C}}^n$. 
	Since $\Pic(\mathbb{A}_{\mathbb{C}}^n)$ is trivial, $p^*$ is surjective. 
	To show $p^*$ injective, let $L \in \Pic(Z)$ be such that $p^*L$ is a trivial bundle on $E$. 
	Since $E$ is Zariski locally trivial over $Z$, we must have $L$ is trivial on $Z$. 
\end{proof}

\begin{lemma}\label{lem-2}
	The natural homomorphism $q_* : \pi_1(E) \longrightarrow \pi_1(Z)$ is an isomorphism. 
\end{lemma}

\begin{proof}
	Since the morphism $p$ is locally trivial, we have an exact sequence of homotopy groups 
	$$ \pi_1(\mathbb{A}_{\mathbb{C}}^n) \longrightarrow \pi_1(E) \stackrel{p_*}{\longrightarrow} \pi_1(Z) 
	\longrightarrow \pi_0(\mathbb{A}_{\mathbb{C}}^n)\,.$$ 
	Since the fibers $\mathbb{A}_{\mathbb{C}}^n$ are connected and contractible, we conclude that $p_*$ 
	is an isomorphism. 
\end{proof}

\begin{lemma}\label{lem-3}
	Let $\iota : U \hookrightarrow Z$ be a Zariski open subset of $Z$, whose complement has codimension 
	at least $2$ in $Z$. Then we have isomorphisms $\Pic(Z) \stackrel{\iota^*}{\longrightarrow} \Pic(U)$ 
	and $\pi_1(U) \stackrel{\simeq}{\longrightarrow} \pi_1(Z)$. 
\end{lemma}

\begin{proof}
	The first isomorphism follows from \cite[Proposition 1.6, p.~126]{Ha2}, and for the second isomorphism, 
	see \cite[p.~42]{Mi}. 
\end{proof}

\subsection{Principal $G$--Higgs bundles}
Let $Y$ be an irreducible smooth projective variety over $\mathbb{C}$. Let $\Omega_Y^1$ be the sheaf of 
differential $1$-forms on $Y$. For any $i \geq 1$, let $\Omega_Y^i = \wedge^i\Omega_Y^1$. Let $G$ be a connected 
reductive affine algebraic group over $k$. Let $\mathfrak{g} = \Lie(G)$ be the Lie algebra of $G$. The adjoint 
action of $G$ on its Lie algebra $\mathfrak{g}$ is denoted by 
$$\ad : G \longrightarrow \End(\mathfrak{g})\,.$$ 

Let $p : E_G \longrightarrow Y$ be a principal $G$--bundle on $Y$. The right $G$--action on $E_G$ and the adjoint 
action of $G$ on $\mathfrak{g}$ produces a $G$--action on $E_G\times \mathfrak{g}$ defined by 
$$(z, \xi)\cdot g = (z\cdot g, \ad(g^{-1})(\xi))\,, \,\, \forall\, (z, \xi) \in G\times\mathfrak{g}\,,\, g \in G\,.$$ 
Then the associated quotient $E_G\times^G \mathfrak{g} := (E_G\times\mathfrak{g})/G$ is a vector bundle $\ad(E_G)$ on 
$Y$, called the \textit{adjoint vector bundle} of $E_G$. A \textit{Higgs field} on $E_G$ is a section 
$\theta \in H^0(Y, \ad(E_G)\otimes\Omega_Y^1)$ such that $\theta\wedge\theta = 0$ in $H^0(Y, \ad(E_G)\otimes\Omega_Y^2)$. 
A \textit{principal $G$--Higgs bundle} on $Y$ is a pair $(E_G, \theta)$ consisting of a principal $G$--bundle 
$E_G$ on $Y$ and a Higgs field $\theta$ on $E_G$. 

\begin{definition}
	A principal $G$--Higgs bundle $(E_G, \theta)$ on $Y$ is said to be \textit{semistable} (respectively, \textit{stable}) 
	if for any reduction $\sigma : U \longrightarrow E_G/P$ of the structure group of $E_G$ to a proper parabolic subgroup 
	$P \subset G$ over an open subset $U \subset Y$ whose complement in $Y$ has codimension at least $2$, such that 
	$\theta \in H^0(X, \ad(E_P)\otimes\Omega_Y^1)$, we have 
	$$\deg\left(\sigma^*(T_{\text{rel}})\right) \geq (\text{respectively, } >) 0\,,$$ 
	where $T_{\text{rel}}$ is the relative tangent bundle of the projection $E_G/P \longrightarrow Y$. 
\end{definition}
Note that, any principal $G$--bundle $E_G$ on $Y$ is a principal $G$--Higgs bundle on $Y$ with Higgs field $\theta = 0$. 
Taking $\theta = 0$ in the above definition, we get the definition of semistability and stability of principal 
$G$--bundles $E_G$ on $Y$. 

\section{Fundamental group and Picard group of $\mathrm{M}_{G, \Higgs}^{\delta}$} 
Fix $\delta \in \pi_1(G)$. 
Let $\mathrm{M}_G^{\delta}$ be the moduli space of semistable holomorphic principal $G$--bundles on $X$ of 
topological type $\delta$. This is a normal complex projective variety (generally non-smooth) of 
dimension $(g-1)\cdot\dim(G) + \dim(Z(G))$. This contains a smooth open subvariety 
$\mathrm{M}_G^{s, \delta}$ parametrizing stable principal $G$--bundles on $X$ of topological type $\delta$. 
Let $\mathrm{M}_{G, \Higgs}^{\delta}$ be the moduli space of semistable holomorphic principal $G$--Higgs 
bundles on $X$ of topological type $\delta$, and let $\mathrm{M}_{G, \Higgs}^{s, \delta}$ be the subvariety 
parametrizing the stable principal $G$--Higgs bundles on $X$ of topological type $\delta$. 
Note that $\mathrm{M}_{G, \Higgs}^{s, \delta}$ is a smooth quasi-projective variety over $\mathbb{C}$. 

$(*)$ Assume that either $g = \text{genus}(X) \geq 3$ or there is no nontrivial homomorphism of $G$ 
onto $\mathrm{PGL}(2, \mathbb{C})$. 

\begin{proposition}\label{prop-1}
	There is an open embedding 
	$$\phi : T^*\mathrm{M}_{G}^{s, \delta} \hookrightarrow \mathrm{M}_{G, \Higgs}^{s, \delta}\,.$$ 
	The complement of the image of $\phi$ in $\mathrm{M}_{G, \Higgs}^{s, \delta}$ has codimension at 
	least $2$. 
\end{proposition}

\begin{proof}
	Let $z \in \mathrm{M}_{G}^{s, \delta}(\mathbb{C})$ be a closed point of the moduli space represented 
	by a stable principal $G$--bundle $E_G$ on $X$ of topological type $\delta$. 
	Since $G$ is reductive, there an isomorphism of $G$--modules 
	$\mathfrak{g} \stackrel{\simeq}{\longrightarrow} \mathfrak{g}^*$. This gives an isomorphism of the 
	adjoint vector bundle $\ad(E_G)$ with its dual $\ad(E_G)^*$. It follows from the 
	deformation theory that the tangent space of $\mathrm{M}_{G}^{s, \delta}$ at $z$ is given by 
	$$T_{z}(\mathrm{M}_{G}^{s, \delta}) \cong H^1(X, \ad(E_G)) \cong H^0(X, \ad(E_G)\otimes\Omega_X^1)^*\,,$$ 
	where the second isomorphism is given by Serre duality. So $T^*_{z}(\mathrm{M}_{G}^{s, \delta})$ 
	is isomorphic to the space of all Higgs fields on $E_G$. 
	Therefore, a closed point of $T^*(\mathrm{M}_{G}^{s, \delta})$ is given by a pair $(E_G, \theta)$, for 
	some $E_G \in \mathrm{M}_{G}^{s, \delta}(\mathbb{C})$ and $\theta \in H^0(X, \ad(E_G)\otimes\Omega_X^1)$. 
	This defines an open embedding 
	\begin{equation}\label{eq-1}
		\phi : T^*(\mathrm{M}_{G}^{s, \delta}) \hookrightarrow \mathrm{M}_{G, \Higgs}^{s, \delta}\,,
	\end{equation} 
	where $T^*(\mathrm{M}_{G}^{s, \delta})$ is the cotangent bundle of $\mathrm{M}_{G}^{s, \delta}$. 
	It follows from \cite[Theorem II.6 (iii), p.~534]{F} that 
	$\mathrm{M}_{G, \Higgs}^{s, \delta}\setminus \phi(T^*(\mathrm{M}_{G}^{s, \delta}))$ has codimension 
	$\geq 2$ in $\mathrm{M}_{G, \Higgs}^{s, \delta}$ (here we are using the condition $(*)$). 
\end{proof}

\begin{theorem}\label{thm-1}
	There are isomorphisms 
	$$\pi_1(\mathrm{M}_{G}^{s, \delta}) \cong \pi_1(\mathrm{M}_{G, \Higgs}^{s, \delta})\,\, 
	\text{ and }\,\, 
	\Pic(\mathrm{M}_{G}^{s, \delta}) \cong \Pic(\mathrm{M}_{G, \Higgs}^{s, \delta})\,.$$ 
\end{theorem}

\begin{proof}
	Consider the open embedding $\phi$ in the Proposition \ref{prop-1}. 
	Then by Lemma \ref{lem-3}, we have isomorphisms 
	\begin{equation}\label{eq-2}
		\pi_1(T^*(\mathrm{M}_{G}^{s,\delta})) \cong \pi_1(\mathrm{M}_{G, \Higgs}^{s, \delta})\,,
	\end{equation} 
	and
	\begin{equation}\label{eq-3}
		\Pic(T^*\mathrm{M}_{G}^{s, \delta}) \cong \Pic(\mathrm{M}_{G, \Higgs}^{s, \delta})\,.
	\end{equation}
	Since $p : T^*\mathrm{M}_{G}^{s, \delta} \longrightarrow \mathrm{M}_{G}^{s, \delta}$ is a vector 
	bundle over a smooth base $\mathrm{M}_{G}^{s, \delta}$, from Lemma \ref{lem-2} we have an isomorphism 
	\begin{equation}\label{eq-4}
		\pi_1(\mathrm{M}_{G}^{s, \delta}) \cong \pi_1(T^*\mathrm{M}_{G}^{rs, \delta})\,.
	\end{equation} 
	Similarly, from Lemma \ref{lem-2}, we have an isomorphism of Picard groups 
	\begin{equation}\label{eq-5}
		p^* : \Pic(\mathrm{M}_{G}^{s, \delta}) \longrightarrow \Pic(T^*\mathrm{M}_{G}^{s, \delta})\,.
	\end{equation} 
	Then from \eqref{eq-2} and \eqref{eq-4}, we have an isomorphism 
	$\pi_1(\mathrm{M}_{G}^{s, \delta}) \cong \pi_1(\mathrm{M}_{G, \Higgs}^{s, \delta})\,.$ 
	Similarly, combining \eqref{eq-3} and \eqref{eq-5}, we have 
	$\Pic(\mathrm{M}_{G}^{s, \delta}) \cong \Pic(\mathrm{M}_{G, \Higgs}^{s, \delta})\,.$ 
\end{proof}

\begin{corollary}\label{cor-1}
	There are isomorphisms 
	$$\pi_1(\mathrm{M}_{G}^{\delta}) \cong \pi_1(\mathrm{M}_{G, \Higgs}^{\delta})\,\, 
	\textnormal{ and }\,\, 
	\Pic(\mathrm{M}_{G}^{\delta}) \cong \Pic(\mathrm{M}_{G, \Higgs}^{\delta})\,. $$
\end{corollary}

\begin{proof}
	It follows from \cite[Proposition 3.4]{KN}\label{prop-2} and openness of stability that 
	$\mathrm{M}_{G}^{s, \delta}$ is a smooth open subscheme of $\mathrm{M}_{G}^{\delta}$, and the 
	complement $\mathrm{M}_{G}^{\delta}\setminus \mathrm{M}_{G}^{s, \delta}$ has codimension $\geq 2$. 
	Then from Lemma \ref{lem-3}, we have isomorphisms 
	\begin{equation}\label{eq-6}
		\pi_1(\mathrm{M}_{G}^{s, \delta}) \cong \pi_1(\mathrm{M}_{G}^{\delta})\,\, 
		\text{ and }\,\, 
		\Pic(\mathrm{M}_{G}^{s, \delta}) \cong \Pic(\mathrm{M}_{G}^{\delta})\,.
	\end{equation} 
	Similarly, $\mathrm{M}_{G, \Higgs}^{\delta}$ contains $\mathrm{M}_{G, \Higgs}^{s, \delta}$ as a 
	smooth open subscheme such that the complement 
	$\mathrm{M}_{G, \Higgs}^{\delta}\setminus \mathrm{M}_{G, \Higgs}^{s, \delta}$ has codimension $\geq 2$ 
	(see \cite[Theorem II.6]{F}). 
	Then from Lemma \ref{lem-3}, we have isomorphisms 
	\begin{equation}\label{eq-7}
	\pi_1(\mathrm{M}_{G, \Higgs}^{s, \delta}) \cong \pi_1(\mathrm{M}_{G, \Higgs}^{\delta})\,\, 
	\text{ and }\,\, 
	\Pic(\mathrm{M}_{G, \Higgs}^{s, \delta}) \cong \Pic(\mathrm{M}_{G, \Higgs}^{\delta})\,. 
	\end{equation} 
	Then the corollary follows from Theorem \ref{thm-1} and the above isomorphisms in \eqref{eq-6} and \eqref{eq-7}. 
\end{proof}

\section*{Acknowledgement}
	We would like to thank Prof. Indranil Biswas for helpful discussions in preparation of this article.

\end{document}